\documentclass[12pt,twoside,reqno]{amsart}
\usepackage[colorlinks=true,citecolor=blue]{hyperref}
\usepackage{mathptmx, amsmath, amssymb, amsfonts, amsthm, mathptmx, enumerate, color}
\usepackage{graphicx}
\usepackage{epstopdf}

\newtheorem{theorem}{Theorem}[section]

\newtheorem{lemma}{Lemma}[section]

\theoremstyle{definition}

\numberwithin{equation}{section}

\begin{document}
\setcounter{page}{1}

\vspace*{1.0cm}
\title[Norm of the Ces\`aro operator between some spaces of analytic functions]
{Norm of the Ces\`aro operator between some spaces of analytic functions}
\author[ S. Ye, B. Ji, Q. Zheng]{ Shanli Ye$^{*}$,  Bin Ji, Qisong Zheng}
\maketitle
\vspace*{-0.6cm}

\begin{center}
{\footnotesize {\it

School of Science, Zhejiang University of Science and Technology, Hangzhou 310023, China.

}}\end{center}

\vskip 4mm {\small\noindent {\bf Abstract.}
In this paper, we determine the exact norm of the Ces\`aro operator $\mathcal{C}$ on the Korenblum space $H^\infty_\alpha$ for $0 < \alpha \leq \frac12$ and on the logarithmically weighted space $H^\infty_{\alpha,\log}$ for $0 < \alpha < 1$. Moreover, we compute its norm when acting from $H^\infty_{\alpha,\log}$ to $H^\infty_\alpha$. Finally, we establish lower and upper bounds for the norm of $\mathcal{C}$ on the $\alpha$-Bloch space $\mathcal{B}^\alpha$ for $\alpha > 1$, and from the Hardy space $H^\infty$ to $\mathcal{B}^\alpha$ for $\alpha\geq 1$.

 }

\renewcommand{\thefootnote}{}
\footnotetext{ $^*$Corresponding author.
\par
E-mail addresses:  slye@zust.edu.cn (S. Ye).
\par
}

\section{Introduction}\label{s1}
Let $(a) = \{a_k\}_{k=0}^{\infty}$ denote a sequence of complex numbers. The classical  Ces\`aro operator, which acts on sequences of this form, is defined by the relation
\begin{align*}
\mathcal{C}((a)) := \left( \frac{1}{n+1} \sum{k=0}^{n} a_k \right)_{n=0}^{\infty}.
\end{align*}

The boundedness of this operator on the $\ell^p$ spaces was already established in the 1920s, through the work of Hardy \cite{Har1} and Landau \cite{Lan}.

An alternative interpretation of the  Ces\`aro operator is as an operator on functions defined on the unit open disc $\mathbb{D} \subset \mathbb{C}$. More precisely, for an analytic function $f \in H(\mathbb{D})$ with Taylor expansion $f(z) = \sum_{k=0}^{\infty} a_k z^k$ ($z \in \mathbb{D}$), the  Ces\`aro operator $\mathcal{C}: H(\mathbb{D}) \to H(\mathbb{D})$ is defined by

\begin{align*}
\mathcal{C}(f)(z) := \sum_{n=0}^{\infty} \left( \frac{1}{n+1} \sum_{k=0}^{n} a_k \right) z^n = \int_0^1 \frac{f(tz)}{1 - tz} dt.
\end{align*}

The boundedness and compactness of the Ces\`aro operator have been the subject of extensive research in complex and functional analysis, as documented in \cite{Dan,Gal,Mia,Xia}. Initial studies of its boundedness on Hardy spaces $H^p$ for $1 < p < \infty$ drew on Hardy's results concerning Fourier series \cite{Har2} and M. Riesz's theorem on conjugate functions \cite[Theorem 4.1]{Dur1}. By employing the theory of composition operator semigroups, Siskakis \cite{Sis3} provided an alternative demonstration of this boundedness and extended the investigation to the case $p = 1$ in \cite{Sis2}. A separate proof for the $p = 1$ case was later contributed by Giang \cite{Gia}. Subsequently, Miao \cite{Mia} established that the Ces\`aro operator remains bounded on $H^p$ for all exponents $0 < p < 1$.

  Recently, Galanopoulos, Girela, and Merch$\acute{a}$n \cite{Gal} introduced a Ces\`aro-like operator $\mathcal{C}_{\mu}$, which is a natural generalization of the classical  Ces\`aro operator $\mathcal{C}$. They systematically  studied this operator acting on  various spaces of analytic functions, such as Hardy spaces, Bergman spaces, and Bloch spaces.  Over the last two decades, several other generalized forms of the classical Ces\`aro operator have been introduced and studied; for these, the interested reader is referred to \cite{Aba, Agr,Alb, And,Dai, Gal1, Nai, Ste}.

  However, there are relatively few works on the exact norm computation of the classical Ces\`aro operator. The main known results in this direction are due to Siskakis. In \cite{Sis3}, he established that $\|\mathcal{C}\|_{H_p} = p$ for $p \geq 2$, while for $1 \leq p < 2$, the norm satisfies $p \leq \|\mathcal{C}\|_{H_p} \leq 2$. In \cite{Sis1}, he showed that $\|\mathcal{C}\|_{A_p} = p/2$ for $p \geq 4$, and $p/2 \leq \|\mathcal{C}\|_{A_p} \leq 2$ for $1 \leq p < 4$. In \cite{Dan}, Danikas and Siskakis also obtained $\|\mathcal{C}\|_{H^{\infty} \to \mathrm{BMOA}} = 1 + \pi/\sqrt{2}$.

In this article, we study the norm of $\mathcal{C}$ acting between certain spaces of analytic functions. Our paper is organized as follows. In Sect.~2, we introduce some notation. In Sect.~3, we determine the exact value of the norm of $\mathcal{C}$ on the Korenblum space $H^\infty_\alpha$  for $0<\alpha\leq \frac12$, which is $\frac{1}{\alpha}$. In Sect.~4, for $0<\alpha<1$, we calculate the exact value of the norm from the logarithmically weighted Korenblum space $H^\infty_{\alpha,\log}$ to the Korenblum space $H^\infty_\alpha$. In Sect.~5, we calculate the exact value of  the norm  on  the logarithmically weighted Korenblum space $H^\infty_{\alpha,\log}$. In Sect.~6, we obtain both the lower and upper bounds of the norm on $\alpha$-Bloch space $\mathcal{B}^\alpha$.  In Sect.~7,  we offer both the lower and upper bounds of  the norm of the Ces\`aro operator from the Hardy space  $H^\infty$ to  $\alpha$-Bloch spaces, show that  $\mathcal{C}: H^\infty_\alpha\to \mathcal{B}^{\alpha}$ is not bounded when $0<\alpha<1$.

\section{Notation Preliminaries}\label{s1}

Let $\mathbb{D}$ denote the open unit disk of the complex plane $\mathbb{C}$, and let $H(\mathbb{D})$ denote the set of all analytic functions in $\mathbb{D}$.

Recall that for $0 < p < \infty$, the Hardy space $H^p$ consists of all analytic functions $f \in H(\mathbb{D})$ satisfying
$$\|f\|_{H^p}=\sup_{0\leq r <1} M_p(r,f)<\infty,$$
where
$$M_p(r,f)=\left( \frac{1}{2\pi}\int_0^{2\pi} |f(re^{it})|^p dt\right)^{\frac{1}{p}}, \quad 0<p<\infty;$$
$$M_\infty(r,f)=\sup_{0\leq t<2\pi}|f(re^{it})|.$$
We refer to \cite{Dur1} for the notation and results regarding Hardy spaces.

For $0 <\alpha<1 $, the Korenblum space $H^\infty_\alpha$ is the space of all functions $f\in H(\mathbb{D} )$ such that
$$\|f\|_{H^\infty_\alpha}=\sup_{z\in\mathbb{D}}(1-|z|^2)^\alpha|f(z)|<\infty.$$
Next, we present the definition of the weighted Korenblum space, which we introduced in reference \cite{Hu},
 for $0 < \alpha<1 $, the logarithmically weighted Korenblum spaces $H^\infty_{\alpha,\log}$ as the set of all $f\in H(\mathbb{D} )$
 such that
$$\|f\|_{H^\infty_{\alpha,\log}}\overset{def}{=} \sup_{z\in\mathbb{D}}(1-|z|^2)^\alpha\log\frac{2e^\frac{1}{\alpha}}{1-|z|^2}|f(z)|<\infty.$$
It is easily verified that $H^\infty \subsetneqq H^\infty_{\alpha,\log} \subsetneqq H^\infty_\alpha$.

For $0<\alpha<\infty$, the $\alpha$-Bloch space $\mathcal{B}^{\alpha}$  consists of those functions $f\in H(\mathbb{D})$ with
$$\|f\|_{\alpha*}=\sup_{z\in\mathbb{D}}(1-|z|^2)^\alpha|f'(z)|<\infty.$$
It is easy to check that $\|\ \|_\mathcal{^{\alpha*}}$ is a complete semi-norm on $\mathcal{B^{\alpha}}$, and $\mathcal{B^{\alpha}}$ can be made into a Banach space by introducing the norm$$\|f\|_\mathcal{B\alpha}=|f(0)|+\|f\|_{\alpha*}.$$

We can see that $\mathcal{B}^1$ is the classical Bloch space $\mathcal{B}$.  We mention \cite{Pom,Zhu} as general references for the classical Bloch space and  the $\alpha$-Bloch spaces.

For an analytic function $f(z) = \sum_{n=0}^{\infty} a_n z^n$ on the unit disk $\mathbb{D}$, the image $\mathcal{C}(f)$ is also analytic on $\mathbb{D}$ and admits several equivalent representations(See \cite{Sis3}). In particular, it can be expressed as:

\begin{align} \label{eq:cesaro-representations}
\mathcal{C}(f)(z) &= \sum_{n=0}^{\infty} \left( \frac{1}{n+1} \sum_{k=0}^{n} a_k \right) z^n \notag \\
&= \int_0^1 \frac{f(tz)}{1-tz} \, dt \notag \\
&= \frac{1}{z} \int_0^z \frac{f(\xi)}{1-\xi} \, d\xi.
\end{align}

By a change of variable in the integral representation, the Ces\`aro operator can be rewritten in terms of a family of weighted composition operators. Specifically, we have:
\begin{align} \label{eq:weighted-composition}
\mathcal{C}(f)(z) = \int_0^{\infty} S_t f(z) \, dt,
\end{align}
where
\begin{align}
S_t f(z) = w_t(z) f(\phi_t(z)), \quad
w_t(z) = \frac{e^{-t}}{1 - (1 - e^{-t})z}, \quad
\phi_t(z) = \frac{e^{-t}z}{1 - (1 - e^{-t})z}. \notag
\end{align}
Differentiating under the integral sign yields the derivative of $\mathcal{C}(f)$:
\begin{align} \label{eq:derivative}
\mathcal{C}(f)'(z) = \int_0^{\infty} \left[ \frac{e^{-t}(1 - e^{-t})}{(1 - (1 - e^{-t})z)^2} f(\phi_t(z))
+ \frac{e^{-2t}}{(1 - (1 - e^{-t})z)^3} f'(\phi_t(z)) \right] dt.
\end{align}

\section{Norm estimates of the Ces\`aro operator $\|\mathcal{C}\|_{H_\alpha^\infty\rightarrow H^\infty_\alpha}$}

In this section, we establish norm estimates for the Ces\`aro operator acting on the Korenblum space $H^\infty_\alpha$.

\begin{theorem}
	For  $0<\alpha\leq\frac{1}{2}$, the Ces\`aro operator $\mathcal{C}$ is bounded on Korenblum space $H_\alpha^\infty$, and its norm satisfies
	\begin{align*}
		\|\mathcal{C}\|_{H_\alpha^\infty\rightarrow H^\infty_\alpha}=\frac{1}{{\alpha}}.
	\end{align*}
\end{theorem}

\begin{proof}

First, we consider the lower bound of $\|\mathcal{C}\|_{H_\alpha^\infty \to H_\alpha^\infty}$. Let $0 < \alpha < 1$ and $z \in \mathbb{D}$. Define
\[
f_\alpha(z) = \frac{1}{(1 - z^2)^\alpha}.
\]
On one hand, we have the estimate
\[
\|f_\alpha\|_{H_\alpha^\infty} = \sup_{z \in \mathbb{D}} \frac{(1 - |z|^2)^\alpha}{|1 - z^2|^\alpha} \leq \sup_{z \in \mathbb{D}} \frac{(1 - |z|^2)^\alpha}{(1 - |z|^2)^\alpha} = 1.
\]
On the other hand, for $r \in (0,1)$, it holds that
\[
\lim_{r \to 1^-} |f_\alpha(r)|(1 - r^2)^\alpha = 1,
\]
and we obtain $\|f_\alpha\|_{H_\alpha^\infty} = 1$.

Now,
\begin{align*}
\|\mathcal{C}\|_{H_\alpha^\infty \to H_\alpha^\infty} &\geq \frac{\|\mathcal{C}(f_\alpha)\|_{H_\alpha^\infty}}{\|f_\alpha\|_{H_\alpha^\infty}} \\
&= \sup_{z \in \mathbb{D}} (1 - |z|^2)^{\alpha} \left| \int_0^{\infty} S_t f_\alpha(z)  dt \right| \\
&= \sup_{z \in \mathbb{D}} (1 - |z|^2)^{\alpha} \left| \int_{0}^{\infty} \frac{e^{-t}}{1 - (1 - e^{-t})z} \cdot \frac{1}{(1 - (\phi_{t}(z))^2)^{\alpha}}  dt \right| \\
&\geq \sup_{0 \leq r < 1} \int_{0}^{\infty} \frac{(1 + r)^\alpha e^{-t} \left(1 - (1 - e^{-t})r \right)^{2\alpha - 1}}{ \left(1 - (1 - 2e^{-t})r\right)^\alpha} dt.
\end{align*}
Letting $r \to 1^{-}$, we obtain
\[
\lim_{r \to 1^-} \int_{0}^{\infty} \frac{(1 + r)^\alpha e^{-t} \left(1 - (1 - e^{-t})r \right)^{2\alpha - 1}}{ \left(1 - (1 - 2e^{-t})r\right)^\alpha} dt = \int_{0}^{\infty} e^{-\alpha t} dt = \frac{1}{\alpha}.
\]

 Next, we derive the upper bound.

Let $f\in H_\alpha^\infty$ with $0<\alpha\leq \frac12$. Using the estimate $|\phi_{t}^{\prime}(z)|=|\frac{e^{-t}}{(1-(1-e^{-t})z)^2}| \leq e^t$
and the Schwarz-Pick lemma, we obtain that
\begin{align*}
	\|S_{t}(f)\|_{H_{\alpha}^{\infty}} &= \sup_{z \in \mathbb{D}} |S_{t}(f)(z)| (1-|z|^{2})^{\alpha} \\
	& = \sup_{z \in \mathbb{D}} \sqrt{e^{-t}} |\phi_{t}^{\prime}(z)|^{1/2} |f(\phi_{t}(z))| (1-|z|^{2})^{\alpha} \\
	&= \sup_{z \in \mathbb{D}}  \sqrt{e^{-t}} |\phi_{t}^{\prime}(z)|^{1/2-\alpha} |\phi_{t}^{\prime}(z)|^{\alpha} |f(\phi_{t}(z))| (1-|z|^{2})^{\alpha} \\
	& \leq  \sqrt{e^{-t}} (e^t)^{1/2-\alpha} \sup_{z \in \mathbb{D}} |f(\phi_{t}(z))| (1-|z|^{2})^{\alpha} |\phi_{t}^{\prime}(z)|^{\alpha} \\
	& \leq e^{-\alpha t} \sup_{z \in \mathbb{D}} |f(\phi_{t}(z))| (1-|\phi_{t}(z)|^{2})^{\alpha}\\
	&\leq e^{-\alpha t} \|f\|_{H_{\alpha}^{\infty}}.
\end{align*}

Then,
\begin{align*}
	\|\mathcal{C}(f)\|_{H_\alpha^\infty}&=\sup_{z\in\mathbb{D}}(1-|z|^2)^{\alpha}\left|\int_{0}^{\infty}S_{t}(f)(z) dt\right|\\
	&\leq\int_{0}^{\infty}\sup_{z\in\mathbb{D}}(1-|z|^2)^{\alpha}|S_{t}(f)(z) |dt\\
	&\leq\int_{0}^{\infty}	\|S_{t}(f)\|_{H_{\alpha}^{\infty}}dt\\
	&\leq\int_{0}^{\infty}e^{-\alpha t} \|f\|_{H_{\alpha}^{\infty}}dt\\
	&=\frac{1}{\alpha}\|f\|_{H_{\alpha}^{\infty}}.
\end{align*}	

Therefore, for $0<\alpha\leq\frac{1}{2}$,
\[
\|\mathcal{C}\|_{H_\alpha^\infty\rightarrow H^\infty_\alpha}=\frac{1}{\alpha}.
\]

This completes the proof of the theorem.

\end{proof}

\section{Norm estimates of the Ces\`aro operator $\|\mathcal{C}\|_{H^\infty_{\alpha,\log} \rightarrow H^\infty_\alpha}$}

Since $H^\infty_{\alpha,\log} \subsetneqq H^\infty_\alpha$, for $0<\alpha<1$, and considering that the norm of the Ces\`aro operator $\mathcal{C}$  on the Korenblum space is bounded in this range, we can conclude that $\mathcal{C}$ is a bound operator acting from $H^\infty_{\alpha,\log}$ into $H^\infty_\alpha$. In this section,  we aim to derive norm estimates for the Ces\`aro operator as it acts from $H^\infty_{\alpha,\log}$ into $H^\infty_\alpha$ for $0 <\alpha<1$.

\begin{theorem}\label{Th3.1}
	For $0<\alpha<1$, then
	$$  \|\mathcal{C}\|_{H^\infty_{\alpha,\log} \rightarrow H^\infty_\alpha}=\sup_{0\leq r<1}\int_{0}^{\infty}\frac{(1+r)^\alpha e^{-t}\left(1-(1-e^{-t})r\right)^{2\alpha-1}}{ \left(1-(1-2e^{-t})r\right)^\alpha\log\frac{2e^\frac{1}{\alpha}}{1-\left(\frac{re^{-t}}{1-(1-e^{-t})r}\right)^2}}dt,$$
and
$$  \|\mathcal{C}\|_{H^\infty_{\alpha,\log} \rightarrow H^\infty_\alpha}\geq\frac{1}{\frac{1}{\alpha}+\log2}	.$$
\end{theorem}
\begin{proof}
	Let $0<\alpha<1$ and $z \in \mathbb{D}$. Define$$f_\alpha(z)=\frac{1}{(1-z^2)^\alpha\log\frac{2e^\frac{1}{\alpha}}{1-z^2}}.$$
	By a simple calculation, we see that $g(x)=x^\alpha\log\frac{2e^\frac{1}{\alpha}}{x}$ is monotonically increasing in $(0,2)$. Since $0\leq |1-z^2|\leq2$, we obtain that
	\begin{align*}
		\|f_\alpha\|_{H^\infty_{\alpha,\log} }
		&=\sup_{z\in\mathbb{D}}(1-|z|^2)^\alpha\log\frac{2e^\frac{1}{\alpha}}{1-|z|^2}\left|\frac{1}{(1-z^2)^\alpha\log\frac{2e^\frac{1}{\alpha}}{1-z^2}}\right|\\
		&\leq\sup_{z\in\mathbb{D}}(1-|z|^2)^\alpha\log\frac{2e^\frac{1}{\alpha}}{1-|z|^2}\frac{1}{(1-|z|^2)^\alpha\log\frac{2e^\frac{1}{\alpha}}{1-|z|^2}}\\
		&=\sup_{0\leq r<1}(1-r^2)^\alpha\log\frac{2e^\frac{1}{\alpha}}{1-r^2}\frac{1}{(1-r^2)^\alpha\log\frac{2e^\frac{1}{\alpha}}{1-r^2}}\\
		&=1.
	\end{align*}
	
	Since
 $$\lim\limits_{r\rightarrow1}|f_\alpha(z)|(1-r^2)^\alpha\log\frac{2e^\frac{1}{\alpha}}{1-r^2}=1,$$  we conclude $\|f_\alpha\|_{H^\infty_{\alpha,\log} }=1$.
	
	The weighted composition operator $S_t$ applied to a function $f_\alpha$ can be written as
	\begin{align*}
		S_tf_\alpha(z)&=\frac{e^{-t}}{1-(1-e^{-t})z}f(\frac{e^{-t}z}{1-(1-e^{-t})z})\\
		&=\frac{e^{-t}}{1-(1-e^{-t})z}\frac{1}{(1-(\frac{e^{-t}z}{1-(1-e^{-t})z})^2)^\alpha\log\frac{2e^\frac{1}{\alpha}}{1-(\frac{e^{-t}z}{1-(1-e^{-t})z})^2}}\\
		&=\frac{ e^{-t}\left(1-(1-e^{-t})z\right)^{2\alpha-1}}{(1-z)^{\alpha} \left(1-(1-2e^{-t})z\right)^\alpha\log\frac{2e^\frac{1}{\alpha}}{1-\left(\frac{e^{-t}z}{1-(1-e^{-t})z}\right)^2}}
	\end{align*}
	Since $\mathcal{C} f_\alpha(z)=\int_0^{\infty} S_t f_\alpha(z) dt$ and $\|f_\alpha\|_{H^\infty_{\alpha,\log} }=1,$ we have
	\begin{align}\label{5.1}
		\|\mathcal{C}\|_{H^\infty_{\alpha,\log} \rightarrow H^\infty_\alpha}&\geq\|\mathcal{C} f_\alpha\|_{H^\infty_\alpha}=\sup_{z\in\mathbb{D}}(1-|z|^2)^\alpha|\mathcal{C} f_\alpha(z)|\notag\\
		&=\sup_{z\in\mathbb{D}}\int_{0}^{\infty}(1-|z|^2)^\alpha\left|\frac{ e^{-t}\left(1-(1-e^{-t})z\right)^{2\alpha-1}}{(1-z)^{\alpha} \left(1-(1-2e^{-t})z\right)^\alpha\log\frac{2e^\frac{1}{\alpha}}{1-\left(\frac{e^{-t}z}{1-(1-e^{-t})z}\right)^2}}\right|dt\notag\\
		&\geq\sup_{0\leq r<1}\int_{0}^{1}(1-r^2)^\alpha\frac{ e^{-t}\left(1-(1-e^{-t})r\right)^{2\alpha-1}}{(1-r)^{\alpha} \left(1-(1-2e^{-t})r\right)^\alpha\log\frac{2e^\frac{1}{\alpha}}{1-\left(\frac{re^{-t}}{1-(1-e^{-t})r}\right)^2}}dt\notag\\
		&=\sup_{0\leq r<1}\int_{0}^{1}\frac{(1+r)^\alpha e^{-t}\left(1-(1-e^{-t})r\right)^{2\alpha-1}}{\left(1-(1-2e^{-t})r\right)^\alpha\log\frac{2e^\frac{1}{\alpha}}{1-\left(\frac{re^{-t}}{1-(1-e^{-t})r}\right)^2}}dt.
	\end{align}
	On the other hand, we have that$$|S_tf(z)|=|w_t(z)f(\phi_t(z))|$$
	\begin{align}\label{5.2}
		&= \lvert w_t(z)\rvert\frac{1}{(1-|\phi_t(z)|^2)^\alpha\log\frac{2e^\frac{1}{\alpha}}{1-|\phi_t(z)|^2}}(1-|\phi_t(z)|^2)^\alpha\log\frac{2e^\frac{1}{\alpha}}{1-|\phi_t(z)|^2}|f(\phi_t(z))|\notag\\
		&\leq \lvert w_t(z)\rvert\frac{1}{(1-|\phi_t(z)|^2)^\alpha\log\frac{2e^\frac{1}{\alpha}}{1-|\phi_t(z)|^2}}\|f\|_{H^\infty_{\alpha,\log} },
	\end{align}
	since $|w_t(z)|=\frac{e^{-t}}{|1-(1-e^{-t})z|}\leq\frac{e^{-t}}{1-(1-e^{-t})|z|}$, $|\phi_t(z)|=\frac{|e^{-t}z|}{|1-(1-e^{-t})z|}\leq\frac{e^{-t}|z|}{1-(1-e^{-t})|z|}$ and the monotonicity of $g(x)$, we obtain
 $$|S_tf(z)|\leq w_t(|z|)\frac{1}{(1-\phi_t(|z|)^2)^\alpha\log\frac{2e^\frac{1}{\alpha}}{1-\phi_t(|z|)^2}}\|f\|_{H^\infty_{\alpha,\log}}.$$
	Therefore, from inequality (\ref{5.2}) we deduce
	\begin{align}\label{eq2.5}
		\|\mathcal{C}f\|_{H^\infty_\alpha}&=\sup_{z\in\mathbb{D}}(1-|z|^2)^\alpha|\int_{0}^{\infty}S_tf(z)dt|\notag\\
		&\leq\sup_{0\leq r<1}(1-r^2)^\alpha\int_{0}^{\infty}\frac{e^{-t}}{1-(1-e^{-t})r}\frac{1}{(1-(\frac{e^{-t}r}{1-(1-e^{-t})r})^2)^\alpha\log\frac{2e^\frac{1}{\alpha}}{1-(\frac{e^{-t}r}{1-(1-e^{-t})r})^2}}dt\|f\|_{H^\infty_{\alpha,\log}}\notag\\
		&=\sup_{0\leq r<1}\int_{0}^{\infty}\frac{(1+r)^\alpha e^{-t}\left(1-(1-e^{-t})r\right)^{2\alpha-1}}{\left(1-(1-2e^{-t})r\right)^\alpha\log\frac{2e^\frac{1}{\alpha}}{1-\left(\frac{re^{-t}}{1-(1-e^{-t})r}\right)^2}}dt\|f\|_{H^\infty_{\alpha,\log}}.
	\end{align}
	Hence,
	\begin{align*}
	\|\mathcal{C}\|_{H^\infty_{\alpha,\log} \rightarrow H^\infty_\alpha}=\sup_{0\leq r<1}\int_{0}^{\infty}\frac{(1+r)^\alpha e^{-t}\left(1-(1-e^{-t})r\right)^{2\alpha-1}}{\left(1-(1-2e^{-t})r\right)^\alpha\log\frac{2e^\frac{1}{\alpha}}{1-\left(\frac{re^{-t}}{1-(1-e^{-t})r}\right)^2}}dt<\infty.
	\end{align*}

Since
\begin{align*}
		&\sup_{0\leq r<1}\int_{0}^{\infty}\frac{(1+r)^\alpha e^{-t}\left(1-(1-e^{-t})r\right)^{2\alpha-1}}{\left(1-(1-2e^{-t})r\right)^\alpha
\log\frac{2e^\frac{1}{\alpha}}{1-\left(\frac{re^{-t}}{1-(1-e^{-t})r}\right)^2}}dt\\
&\geq\sup_{r=0}\int_{0}^{\infty}\frac{(1+r)^\alpha e^{-t}\left(1-(1-e^{-t})r\right)^{2\alpha-1}}{\left(1-(1-2e^{-t})r\right)^\alpha
\log\frac{2e^\frac{1}{\alpha}}{1-\left(\frac{re^{-t}}{1-(1-e^{-t})r}\right)^2}}dt=\frac{1}{\log2e^\frac{1}{\alpha}}\int_{0}^{\infty}e^{-t}dt
	\end{align*}

Thus,
$$\|\mathcal{C}\|_{H^\infty_{\alpha,\log} \rightarrow H^\infty_\alpha}\geq \frac{1}{\log2 +\frac{1}{\alpha}}.$$
	This completes the proof.
\end{proof}

\section{Norm estimates of the Ces\`aro operator $\|\mathcal{C}\|_{H^\infty_{\alpha,\log} \rightarrow H^\infty_{\alpha,\log}}$}
In this section,  we calculate the norm of the Ces\`aro operator $\mathcal{C}$ acting on   the logarithmically weighted Korenblum space $H^\infty_{\alpha,\log}$.

\begin{theorem}
	For $0<\alpha<1$, then the Ces\`aro operator  $\mathcal{C}$ is bounded on the logarithmically weighted Korenblum space  $H^\infty_{\alpha,\log}$. Moreover, the norm of  $\mathcal{C}$ satisfies the following equations:
$$\|\mathcal{C}\|_{H^\infty_{\alpha,\log} \rightarrow H^\infty_{\alpha,\log}}=\sup_{0\leq r<1}\int_{0}^{\infty}\frac{(1+r)^{\alpha}e^{-t}\left(1-(1-e^{-t})r\right)^{2\alpha-1}\log\frac{2e^\frac{1}{\alpha}}{1-r^2}}{\left(1-(1-2e^{-t})r\right)^{\alpha}\log\frac{2e^\frac{1}{\alpha}}{1-\left(\frac{e^{-t}r}{1-(1-e^{-t})r}\right)^2}}dt.$$ and $$\|\mathcal{C}\|_{H^\infty_{\alpha,\log} \rightarrow {H^\infty_{\alpha,\log} }}
	\geq\frac{1}{\alpha}.$$
\end{theorem}

\begin{proof} Let $$f_\alpha(z)=\frac{1}{(1-z^2)^\alpha\log\frac{2e^\frac{1}{\alpha}}{1-z^2}}.$$
From the proof of Theorem 4.1, we know that  $\|f_\alpha\|_{H^\infty_{\alpha,\log} }=1$ and
$$
	S_tf_\alpha(z)=\frac{ e^{-t}\left(1-(1-e^{-t})z\right)^{2\alpha-1}}{(1-z)^{\alpha} \left(1-(1-2e^{-t})z\right)^\alpha\log\frac{2e^\frac{1}{\alpha}}{1-\left(\frac{e^{-t}z}{1-(1-e^{-t})z}\right)^2}}.
$$
Since $\mathcal{C} f_\alpha(z)=\displaystyle\int_0^{\infty} S_t f_\alpha(z) dt$, we deduce
\begin{align*}
	\|\mathcal{C}\|_{H^\infty_{\alpha,\log} \rightarrow H^\infty_{\alpha,\log}}&\geq\|\mathcal{C}f_\alpha\|_{H^\infty_{\alpha,\log}}\\
	&=\sup_{z\in\mathbb{D}}(1-|z|^2)^\alpha\log\frac{2e^\frac{1}{\alpha}}{1-|z|^2}|\mathcal{C}f_\alpha(z)|\\
	&\geq\sup_{0\leq r<1}(1-r^2)^{\alpha}\log\frac{2e^\frac{1}{\alpha}}{1-r^2}\mathcal{C}f_\alpha(r)\\
	&=\sup_{0\leq r<1}\int_{0}^{\infty}\frac{(1+r)^{\alpha}e^{-t}\left(1-(1-e^{-t})r\right)^{2\alpha-1}\log\frac{2e^\frac{1}{\alpha}}{1-r^2}}{\left(1-(1-2e^{-t})r\right)^{\alpha}\log\frac{2e^\frac{1}{\alpha}}{1-(\frac{e^{-t}r}{1-(1-e^{-t})r})^2}} dt.\\
\end{align*}
On the other hand, using the estimate
  $$|S_tf(z)|\leq w_t(|z|)\frac{1}{(1-\phi_t(|z|)^2)^\alpha\log\frac{2e^\frac{1}{\alpha}}{1-\phi_t(|z|)^2}}\|f\|_{H^\infty_{\alpha,\log}},$$
we obtain that
\begin{align*}
	\|\mathcal{C}f\|_{H^\infty_{\alpha,\log}}&=\sup_{z\in\mathbb{D}}(1-|z|^2)^{\alpha}\log\frac{2e^\frac{1}{\alpha}}{1-|z|^2}|\int_{0}^{\infty}S_tf(z)dt|\\
	&\leq\sup_{z\in\mathbb{D}}(1-|z|^2)^{\alpha}\log\frac{2e^\frac{1}{\alpha}}{1-|z|^2}|\int_{0}^{\infty}w_t(|z|)\frac{1}{(1-\phi_t(|z|)^2)^\alpha\log\frac{2e^\frac{1}{\alpha}}{1-\phi_t(|z|)^2}}dt\|f\|_{H^\infty_{\alpha,\log}}\\
	&\leq\sup_{0\leq r<1}\int_{0}^{\infty}\frac{(1+r)^{\alpha}e^{-t}\left(1-(1-e^{-t})r\right)^{2\alpha-1}\log\frac{2e^\frac{1}{\alpha}}{1-r^2}}{\left(1-(1-2e^{-t})r\right)^{\alpha}\log\frac{2e^\frac{1}{\alpha}}{1-(\frac{e^{-t}r}{1-(1-e^{-t})r})^2}} dt\|f\|_{H^\infty_{\alpha,\log}}.\\
\end{align*}
Therefore,
$$\|\mathcal{C}\|_{H^\infty_{\alpha,\log} \rightarrow H^\infty_{\alpha,\log}}=\sup_{0\leq r<1}\int_{0}^{\infty}\frac{(1+r)^{\alpha}e^{-t}\left(1-(1-e^{-t})r\right)^{2\alpha-1}\log\frac{2e^\frac{1}{\alpha}}{1-r^2}}{\left(1-(1-2e^{-t})r\right)^{\alpha}\log\frac{2e^\frac{1}{\alpha}}{1-\left(\frac{e^{-t}r}{1-(1-e^{-t})r}\right)^2}}dt.$$
Now, observe that  $$\lim_{r\rightarrow1^-}\frac{\log\frac{2e^\frac{1}{\alpha}}{1-r^2}}{\log\frac{2e^\frac{1}{\alpha}}{1-\left(\frac{e^{-t}r}{1-(1-e^{-t})r}\right)^2}}=1,$$
Hence,
\begin{align*}
	\|\mathcal{C}\|_{H^\infty_{\alpha,\log} \rightarrow {H^\infty_{\alpha,\log} }}
	&\geq\lim_{r\rightarrow1^-}\int_{0}^{\infty}\frac{(1+r)^{\alpha}e^{-t}\left(1-(1-e^{-t})r\right)^{2\alpha-1}\log\frac{2e^\frac{1}{\alpha}}{1-r^2}}{\left(1-(1-2e^{-t})r\right)^{\alpha}\log\frac{2e^\frac{1}{\alpha}}{1-\left(\frac{e^{-t}r}{1-(1-e^{-t})r}\right)^2}}dt\\
	&=\int_{0}^{\infty}e^{-\alpha t} dt\\
	&=\frac{1}{\alpha}
\end{align*}

\end{proof}

\section{Norm estimates of the Ces\`aro operator$ \|\mathcal{C}\|_{\mathcal{B}^\alpha\rightarrow \mathcal{B}^\alpha}$}

   From \cite{Xia}, we know that the Ces\`aro operator $\mathcal{C}$ is bounded on the $\alpha$-Bloch space $\mathcal{B}^\alpha$ if and only if $1 < \alpha < \infty$. Here, we provide an upper bound and a lower bound for the norm of the Ces\`aro operator when $\alpha > 1$.

 \begin{lemma}
 \cite[Lemma 5.1]{Hu}
 Let $f\in\mathcal{B}^\alpha$, then
  \begin{align*}
	|f(z)|\leq
	\begin{cases}
		&\frac{(1-|z|)^{1-\alpha}-1}{\alpha-1}\|f\|_{\alpha*}+|f(0)|, if  \alpha\not=1,\\
		&\log\frac{1}{1-|z|}\|f\|_{\alpha*}+|f(0)|,  if  \alpha =1.\\
	\end{cases}
\end{align*}

 \end{lemma}
\begin{theorem}
 For the Ces\`aro operator $\mathcal{C}$ acting on the space $\mathcal{B}^\alpha$, the following upper bounds hold for its norm:
\[
\|\mathcal{C}\|_{\mathcal{B}^\alpha \to \mathcal{B}^\alpha} \leq
\begin{cases}
\max \left\{ A, \dfrac{2^\alpha}{\alpha - 1} \right\}, & \text{for } 1 < \alpha \leq 2, \\[2em] 
\max \left\{ A, 2^\alpha \dfrac{2^{\alpha} - \alpha - 1}{(\alpha - 1)^2} \right\}, & \text{for } \alpha > 2,
\end{cases}
\]
where
\[
A = 1 + \left( \frac{2}{2\alpha - 1} \right)^{2\alpha - 1} \alpha^{\alpha} (\alpha - 1)^{\alpha - 1}.
\]

\end{theorem}

\begin{proof}
		Consider the upper bound of  $\|\mathcal{C}\|_{\mathcal{B}^\alpha\rightarrow\mathcal{B}^\alpha}$. Let $f\in \mathcal{B}^\alpha$, we have that
		\begin{align*}
			(\mathcal{C}f)^{\prime}(z) &= \int_{0}^{1} \frac{tf^{\prime}(tz)}{1-tz}  dt + \int_{0}^{1} \frac{tf(tz)}{(1-tz)^{2}}  dt\\
			&=\frac{1}{z^2}\int_{0}^{z} \frac{\xi f^{\prime}(\xi)}{1-\xi} + \frac{\xi f(\xi)}{(1-\xi)^{2}}  d\xi, \quad z \in \mathbb{D}.
		\end{align*}
		We can change the path of integration to
		\begin{align*}
		\xi = \phi_t(z) = \frac{e^{-t}z}{1 - (1 - e^{-t})z} , \quad 0 \leq t < \infty.
		\end{align*}
		Therefore, we obtain that
		\begin{align*}
			(\mathcal{C}f)'(z) =  \int_{0}^{\infty} \frac{e^{-2t}}{1 - \left( 1 - e^{-t} \right) z} \left[ \frac{f'(\phi_{t}(z))}{1 - \left( 1 - e^{-t} \right) z}  + \frac{f(\phi_{t}(z))}{1-z} \right]  dt
		\end{align*}
		Then
		\begin{align*}
		\|\mathcal{C}f(z)\|_{\mathcal{B}^\alpha}&=|\mathcal{C}f(0)|+\sup_{z\in\mathbb{D}}|\mathcal{C}f'(z)|(1-|z|^2)^\alpha\\
		&\leq|f(0)|+\sup_{z\in\mathbb{D}}(1-|z|^2)^\alpha\int_{0}^{\infty} \frac{e^{-2t}}{|1 - ( 1 - e^{-t} ) z|} \left(\frac{|f'(\phi_{t}(z))|}{|1 - ( 1 - e^{-t} ) z|}  + \frac{|f(\phi_{t}(z))|}{|1-z|} \right) dt.
		\end{align*}
		Since $|1-z| \geq 1-|z| = 1-r$, $|f'(\varphi_t(z))| \leq \frac{\|f\|_{\alpha^*}}{\left(1-|\phi_t(z)|^2\right)^\alpha}$, we can apply Lemma 6.1 to obtain the following estimate,
		\begin{align*}
		\|\mathcal{C}f(z)\|_{\mathcal{B}^\alpha} &\leq |f(0)| + \sup_{0 \leq |z| < 1} (1 - |z|^2)^\alpha \int_{0}^{\infty} [ \frac{e^{-2t} \| f \|_{\alpha^*}}{|1 - (1 - e^{-t}) z|^2  (1 - |\phi_t(z)|^2)^\alpha}\\ &+ \frac{e^{-2t} \left[ (1 - |\phi_t(z)|)^{1-\alpha} - 1 \right]  \| f \|_{\alpha^*}}{(\alpha - 1) (1 - |z|)  |1 - (1 - e^{-t}) z|} + \frac{e^{-2t}  |f(0)|}{ (1 - |z|)  |1 - (1 - e^{-t}) z|} ] dt\\
		&\leq\left(1+\sup_{0 \leq r < 1} (1 - r^2)^\alpha \int_{0}^{\infty}\frac{e^{-2t}}{ (1 - r)  (1 - (1 - e^{-t}) r)}  dt\right)|f(0)|\\
		&+\sup_{0 \leq |z| < 1} (1 - |z|^2)^\alpha\int_{0}^{\infty} [ \frac{e^{-2t}}{|1 - (1 - e^{-t}) z|^2  (1 - |\phi_t(z)|^2)^\alpha}\\&+ \frac{e^{-2t} \left[ (1 - |\phi_t(z)|)^{1-\alpha} - 1 \right] }{(\alpha - 1) (1 - |z|)  |1 - (1 - e^{-t}) z|}] \| f \|_{\alpha^*} dt\overset{def}{=}|f(0)|\cdot I+\| f \|_{\alpha^*}\cdot II.
		\end{align*}
		We estimate that
		$$I=1+\sup_{0 \leq r < 1} (1 - r^2)^\alpha \int_{0}^{\infty}\frac{e^{-2t} }{ (1 - r)  (1 - (1 - e^{-t}) r)}dt.$$
		\begin{align*}
		&\leq 1 + \sup_{0 \leq r < 1} (1+r)^{\alpha}  (1 - r)^{\alpha-1} \ \int_{0}^{\infty} e^{-t}  dt \\
		&= 1 + \sup_{0 \leq r < 1} (1 + r)^{\alpha} (1 - r)^{\alpha - 1}.
		\end{align*}
		Define the function $$f(r)=(1 + r)^{\alpha} (1 - r)^{\alpha - 1}.$$
  A straightforward calculation shows that $f(r)$ attains its maximum
		 when $r=\frac{1}{2\alpha-1}$. Therefore,
		 \begin{align}\label{1.1}
		 	1+\sup_{0 \leq r < 1} (1 - r^2)^\alpha \int_{0}^{\infty}\frac{e^{-2t} }{ (1 - r)  (1 - (1 - e^{-t}) r)}  dt\leq1+\left( \frac{2}{2\alpha - 1} \right)^{2\alpha - 1} \alpha^{\alpha} (\alpha - 1)^{\alpha - 1}\overset{def}{=}A.
		 \end{align}

		To bound $II$, we obtain that
		$$II=\sup_{0 \leq |z| < 1} (1 - |z|^2)^\alpha\int_{0}^{\infty}  \frac{e^{-2t} }{|1 - (1 - e^{-t}) z|^2  (1 - |\phi_t(z)|^2)^\alpha}+ \frac{e^{-2t} \left[ (1 - |\phi_t(z)|)^{1-\alpha} - 1 \right] }{(\alpha - 1) (1 - |z|)  |1 - (1 - e^{-t}) z|}dt.$$
		\begin{align*}
			&\leq\sup_{0 \leq |z| < 1} (1 - |z|^2)^\alpha\int_{0}^{\infty}  \frac{e^{-2t} }{|1 - (1 - e^{-t}) z|^2  (1 - |\phi_t(z)|)^\alpha}+ \frac{e^{-2t}  (1 - |\phi_t(z)|)^{1-\alpha}   }{(\alpha - 1) (1 - |z|)  |1 - (1 - e^{-t}) z|}dt\\
			&= \sup_{0 \leq |z| < 1} (1 - |z|^2)^{\alpha}  \int_{0}^{\infty}  \frac{e^{-2t}  | 1 - (1 - e^{-t}) z |^{\alpha - 2}}{\left( | 1 - (1 - e^{-t}) z| - e^{-t} |z| \right)^{\alpha}} + \frac{e^{-2t} |1 - (1 - e^{-t}) z|^{\alpha - 2}}{(\alpha - 1)  (1 - |z|)  \left(| 1 - (1 - e^{-t}) z | - e^{-t} |z| \right)^{\alpha - 1}}  dt\\
			&\leq \sup_{0 \leq |z| < 1} (1 - |z|^2)^{\alpha}  \int_{0}^{\infty}  \frac{e^{-2t}  | 1 - (1 - e^{-t}) z |^{\alpha - 2}}{(  1-|z| )^{\alpha}} + \frac{e^{-2t} |1 - (1 - e^{-t}) z|^{\alpha - 2}}{(\alpha - 1)  (1 - |z|) ^{\alpha}}  dt\\
			&=\sup_{0 \leq |z| < 1} (1+|z|)^{\alpha}\frac{\alpha}{\alpha-1}\int_{0}^{\infty}e^{-2t}  | 1 - (1 - e^{-t}) z |^{\alpha - 2}dt.
		\end{align*}
	Now, we consider the case $1<\alpha \leq2$. We estimate the supremum as follows:
		$$\sup_{0 \leq |z| < 1} (1+|z|)^{\alpha}\frac{\alpha}{\alpha-1}\int_{0}^{\infty}e^{-2t}  | 1 - (1 - e^{-t}) z |^{\alpha - 2}dt$$
		\begin{align}\label{1.2}
		&\leq\sup_{0 \leq r < 1} (1+r)^{\alpha}\frac{\alpha}{\alpha-1}\int_{0}^{\infty}\frac{e^{-2t}}{(1 - (1 - e^{-t})r)^{2-\alpha}}  dt \notag\\
		&=2^\alpha \frac{\alpha}{\alpha-1}\int_{0}^{\infty}\frac{e^{-2t}}{e^{-(2-\alpha)t}}dt \notag\\
		&=\frac{2^\alpha}{\alpha-1}
		\end{align}
	Next, we consider the case  $\alpha > 2$. A similar estimation yields:
		$$\sup_{0 \leq |z| < 1} (1+|z|)^{\alpha}\frac{\alpha}{\alpha-1}\int_{0}^{\infty}e^{-2t}  | 1 - (1 - e^{-t}) z |^{\alpha - 2}dt$$
		\begin{align}\label{1.3}
			&\leq \sup_{0 \leq r < 1} (1+r)^{\alpha}\frac{\alpha}{\alpha-1}\int_{0}^{\infty}\frac{e^{-2t}}{(1 - (e^{-t}-1)r)^{2-\alpha}}  dt\notag\\
			&=2^\alpha \frac{\alpha}{\alpha-1}\int_{0}^{\infty}\frac{e^{-2t}}{(2-e^{-t})^{2-\alpha}}dt\notag\\
			&=2^\alpha \frac{2^{\alpha} - \alpha - 1}{(\alpha - 1)^2}
		\end{align}
		Combining (\ref{1.1}), (\ref{1.2}) and (\ref{1.3}), we finally complete the proof.
		
	\end{proof}

\begin{theorem}
	For $1<\alpha<\infty$, we have
	$$\|\mathcal{C}\|_{\mathcal{B}^\alpha\rightarrow \mathcal{B}^\alpha}\geq\frac{3}{2}$$
\end{theorem}
\begin{proof}
	Let $\alpha\not=1$ and $z\in\mathbb{D}$.Define$$f_\alpha=1$$
	we have the estimate
	$$\|f_\alpha\|_{\mathcal{B}^\alpha}=1$$
	
	According to (\ref{eq:derivative}), we obtain that
	\begin{align*}
		\|\mathcal{C}\|_{\mathcal{B}^\alpha\rightarrow\mathcal{B}^\alpha}&\geq\frac{\|\mathcal{C}f_\alpha\|_{\mathcal{B}^\alpha}}{\|f_\alpha\|_{\mathcal{B}^\alpha}}\\
		&=|\mathcal{C}f(0)|+\sup_{z\in\mathbb{D}}(1-|z|^2)^\alpha|\mathcal{C}f_\alpha'(z)|\\
		&=1+\sup_{z\in\mathbb{D}}(1-|z|^2)^\alpha\big|\int_{0}^{\infty}\frac{e^{-t}(1-e^{-t})}{(1-(1-e^{-t})z)^2} dt\big|\\
		&\geq1+\sup_{0\leq r<1}(1-r^2)^\alpha\int_{0}^{\infty}\frac{e^{-t}(1-e^{-t})}{(1-(1-e^{-t})r)^2} dt\\
		&=1+\lim\limits_{r\rightarrow0}(1-r^2)^\alpha\int_{0}^{\infty}\frac{e^{-t}(1-e^{-t})}{(1-(1-e^{-t})r)^2} dt\\
		&=1+\int_{0}^{\infty}e^{-t}(1-e^{-t}) dt\\
		&=1+\frac{1}{2}\\
		&=\frac{3}{2}.
	\end{align*}
	 At this point, we have completed all the proof.
\end{proof}

\section{Norm of Ces\`aro operator $\|\mathcal{C}\|_{H^\infty\rightarrow \mathcal{B}^\alpha}$}

In this section we offer both the lower and upper bounds of  the norm of the Ces\`aro operator from the Hardy space  $H^\infty$ to  $\alpha$-Bloch spaces
$\mathcal{B}^\alpha$ with $\alpha\geq 1$.
\begin{theorem} For $\alpha\geq 1$, we obtain that
\begin{align*}
		3 &\leq \|\mathcal{C}\|_{H^\infty\rightarrow \mathcal{B}^\alpha}\leq 4~~~  \mbox{if}~~~ \alpha=1,\\
		\frac 32 &\leq \|\mathcal{C}\|_{H^\infty\rightarrow \mathcal{B}^\alpha}\leq 4~~~ \mbox{if}~~~ \alpha>1.
	\end{align*}
Moreover, $\mathcal{C}$ is not bounded from $H^\infty$ to $\mathcal{B}^\alpha$ when $0<\alpha< 1$.
	
	\end{theorem}
	
		\begin{proof}
 let $f=1\in H^\infty$, then
  $$\|f\|_{\infty}=1,~~\mathcal{C}(1)(z)=\frac{1}{z}\log\frac{1}{1-z},~~\mathcal{C}(1)(0)=1.$$
	We have that
	\begin{align*}
		\|\mathcal{C}\|_{H^\infty\rightarrow \mathcal{B}^{\alpha}}&\geq|\mathcal{C}(f)(0)|+\|\mathcal{C}(f)(z)\|_*\\
		&=1+\sup_{z\in\mathbb{D}}(1-|z|^2)^{\alpha}|\mathcal{C}(1)'(z)|\\
		&=1+\sup_{z\in\mathbb{D}}(1-|z|^2)^{\alpha}\big|\frac{1}{z(1-z)}-\frac{1}{z^2}\log\frac{1}{1-z} \big|\\
		&\geq1+\sup_{0\leq r<1}(1-r^2)^{\alpha}\big|\frac{1}{r(1-r)}-\frac{1}{r^2}\log\frac{1}{1-r}\big|.
	\end{align*}
Let $f(r)=\displaystyle(1-r^2)^{\alpha}\big|\frac{1}{r(1-r)}-\frac{1}{r^2}\log\frac{1}{1-r}\big|$.  By a simple calculation, we have that
\begin{align*}
		\sup_{0\leq r<1}f(r)&\geq \lim_{r\to 0^+} f(r)=+\infty~~~  \mbox{if}~~~ 0<\alpha< 1,\\
		\sup_{0\leq r<1}f(r) &= \lim_{r\to 1} f(r)=2~~~ \mbox{if}~~~ \alpha=1.\\
  	\end{align*}
Therefore, $\mathcal{C}$ is not bounded from $H^\infty$ to $\mathcal{B}^\alpha$ when $0<\alpha< 1$.
For the case $\alpha>1$, similar to the proof of Theorem 6.2, we have that

\begin{align*}
		\|\mathcal{C}\|_{H^\infty\rightarrow \mathcal{B}^\alpha}&\geq|\mathcal{C}(1)(0)|+\sup_{z\in\mathbb{D}}(1-|z|^2)^\alpha|\mathcal{C}(1)'(z)|\\
		&=1+\sup_{z\in\mathbb{D}}(1-|z|^2)^\alpha\big|\int_{0}^{\infty}\frac{e^{-t}(1-e^{-t})}{(1-(1-e^{-t})z)^2} dt\big|\\
		&\geq1+\sup_{0\leq r<1}(1-r^2)^\alpha\int_{0}^{\infty}\frac{e^{-t}(1-e^{-t})}{(1-(1-e^{-t})r)^2} dt\\
				&=\frac{3}{2}.
		\end{align*}
This completes the proof of the lower bound. It remains to prove the upper bound for the case $\alpha\geq 1$.

 Let $f \in H^{\infty}$. Then
			\begin{align}{\label{7.1}}
				\mathcal{C}(f)'(z) = \int_{0}^{1} \frac{t f'(t z)}{1 - t z}  dt + \int_{0}^{1} \frac{t f(t z)}{(1 - t z)^{2}}  dt, \quad z \in \mathbb{D}.
			\end{align}
			
			By \cite[Proposition 5.1]{Zhu}, we have that
			\begin{align}{\label{7.2}}
				(1 - |z|^{2}) |f'(z)| \leq \|f\|_{\infty}.
			\end{align}
			for all $f \in H^{\infty}$ and $z \in \mathbb{D}$.		

		Using (\ref{7.1}) and (\ref{7.2}), we obtain that
		\begin{align*}
			\|\mathcal{C}f\|_{\mathcal{B}^\alpha}&=|\mathcal{C}f(0)|+\sup_{z\in\mathbb{D}}(1-|z|^2)^\alpha|(\mathcal{C}f)'(z)| \\
			&=|\int_{0}^{1}f(0)dt|+\sup_{z\in\mathbb{D}}(1-|z|^2)^\alpha|(\mathcal{C}f)'(z)| \\
			&=\left|\int_{0}^{1} f(0)dt\right|+\sup _{z\in\mathbb{D}}\left(1-|z|^{2}\right)^{\alpha}\left|\int_{0}^{1}\frac{tf^{\prime}(tz)}{1-tz}dt+\int_{0}^{1}\frac{tf(tz)}{(1-tz)^{2}}dt\right| \\
			&\leq\left(1+\sup_{0\leq r<1}\left(1-r^{2}\right)^{\alpha}\int_{0}^{1}\frac{2t+rt^{2}}{(1-tr)^{2}(1+tr)}dt\right)\|f\|_{\infty} \\
			&=\left(1+\sup_{0<r<1}\frac{(1-r^{2})^{\alpha}}{r^2}\left\{\frac{3r}{2(1-r)}-\frac{1}{4}\ln\frac{1+r}{(1-r)^5}\right\}\right)\|f\|_{\infty}.
		\end{align*}
		Let $h(r)=\displaystyle\frac{(1-r^{2})}{r^2}\left\{\frac{3r}{2(1-r)}-\frac{1}{4}\ln\frac{1+r}{(1-r)^5}\right\}, 0< r<1.$
		 We find that it can be shown that
		$$h(r)=1+r+\sum_{n=2}^{\infty} \left( \frac{5}{2n(n+2)} + \frac{(-1)^{n-1}}{2n(n+2)} \right) r^n,$$
				then
		\begin{align*}
			h'(r)=1+\sum_{n=2}^{\infty} \left( \frac{5}{2(n+2)} + \frac{(-1)^{n-1}}{2(n+2)} \right) r^{n-1}>0.
		\end{align*}
		It implies that $h(r)$  is increasing and  $\sup_{0<r<1}h(r)=\lim_{r\rightarrow1^-} h(r)=3$.
		
Hence, we obtain that
	$$
\|\mathcal{C}f\|_{\mathcal{B}^\alpha}\leq 1+\sup_{0<r<1}h(r)(1-r^2)^{\alpha-1}\leq 1+3\sup_{0<r<1}(1-r^2)^{\alpha-1}=4.
 $$		

		We finish the proof of the theorem.

\end{proof}

\end{document}